\documentclass[a4paper,oneside,english]{amsart}
\usepackage[T1]{fontenc}
\usepackage[latin9]{inputenc}
\usepackage{array}
\usepackage{multirow}
\usepackage{amstext}
\usepackage{amsthm}
\usepackage{amssymb}

\makeatletter

\pdfpageheight\paperheight
\pdfpagewidth\paperwidth

\providecommand{\tabularnewline}{\\}

\numberwithin{equation}{section}
\numberwithin{figure}{section}
\theoremstyle{plain}
\newtheorem{thm}{\protect\theoremname}
\theoremstyle{plain}
\newtheorem{prop}[thm]{\protect\propositionname}

\makeatother

\usepackage{babel}
\providecommand{\propositionname}{Proposition}
\providecommand{\theoremname}{Theorem}

\begin{document}
\title{Congruence Properties of Indices of Triangular Numbers Multiple of
Other Triangular Numbers}
\author{Vladimir PLETSER}
\address{European Space Agency (ret.)}
\email{PletserVladimir@gmail.com}
\begin{abstract}
It is known that, for any positive non-square integer multiplier $k$,
there is an infinity of multiples of triangular numbers which are
triangular numbers. We analyze the congruence properties of the indices
$\xi$ of triangular numbers that are multiples of other triangular
numbers. We show that the remainders in the congruence relations of
$\xi$ modulo $k$ come always in pairs whose sum always equal $\left(k-1\right)$,
always include 0 and $\left(k-1\right)$, and only 0 and $\left(k-1\right)$
if $k$ is prime, or an odd power of a prime, or an even square plus
one or an odd square minus one or minus two. If the multiplier $k$
is twice the triangular number of $n$, the set of remainders includes
also $n$ and $\left(n^{2}-1\right)$ and if $k$ has integer factors,
the set of remainders include multiples of a factor following certain
rules. Finally, algebraic expressions are found for remainders in
function of $k$ and its factors. Several exceptions are noticed and
superseding rules exist between various rules and expressions of remainders.
This approach allows to eliminate in numerical searches those $\left(k-\upsilon\right)$
values of $\xi_{i}$ that are known not to provide solutions, where
$\upsilon$ is the even number of remainders. The gain is typically
in the order of $k/\upsilon$, with $\upsilon\ll k$ for large values
of $k$.
\end{abstract}


\maketitle

\section{Introduction\label{sec1:Introduction}}

Triangular numbers $T_{t}=\frac{t\left(t+1\right)}{2}$ are one of
the figurate numbers enjoying many properties; see, e.g., \cite{key-1,key-2}
for relations and formulas. Triangular numbers $T_{\xi}$ that are
multiples of other triangular number $T_{t}$ 
\begin{equation}
T_{\xi}=kT_{t}\label{eq:1}
\end{equation}
are investigated. Only solutions for $k>1$ are considered as the
cases $k=0$ and $k=1$ yield respectively $\xi=0$ and $\xi=t,\forall t$.
Accounts of previous attempts to characterize these triangular numbers
multiple of other triangular numbers can be found in \cite{key-3,key-4,key-5,key-6,key-7,key-8,key-9}.
Recently, Pletser showed \cite{key-9} that, for non-square integer
values of $k$, there are infinitely many solutions that can be represented
simply by recurrent relations of the four variables $t,\xi,Tt$ and
$T_{\xi}$, involving a rank $r$ and parameters $\kappa$ and $\gamma$,
which are respectively the sum and the product of the $\left(r-1\right)^{\text{th}}$
and the $r^{\text{th}}$ values of $t$. The rank $r$ is being defined
as the number of successive values of $t$ solutions of (\ref{eq:1})
such that their successive ratios are slowly decreasing without jumps.

In this paper, we present a method based on the congruent properties
of $\xi\left(\text{mod\,}\ensuremath{k}\right)$, searching for expressions
of the remainders in function of $k$ or of its factors. This approach
accelerates the numerical search of the values of $t_{n}$ and $\xi_{n}$
that solve (\ref{eq:1}), as it eliminates values of $\xi$ that are
known not to provide solutions to (\ref{eq:1}). The gain is typically
in the order of $k/\upsilon$ where $\upsilon$ is the number of remainders,
which is usually such that $\upsilon\ll k$. 

\section{Rank and Recurrent Equations\label{sec2:Rank-and-Recurrent} }

Sequences of solutions of (\ref{eq:1}) are known for $k=2,3,5,6,7,8$
and are listed in the Online Encyclopedia of Integer Sequences (OEIS)
\cite{key-10}, with references given in Table \ref{tab1:OEIS--references}.

\begin{table}
\caption{\label{tab1:OEIS--references}OEIS \cite{key-10} references of sequences
of integer solutions of (\ref{eq:1}) for $k=2,3,5,6,7,8$}

\centering{}%
\begin{tabular}{ccccccc}
\hline 
$k$ & 2 & 3 & 5 & 6 & 7 & 8\tabularnewline
\hline 
\hline 
$t$ & A053141 & A061278 & A077259 & A077288 & A077398 & A336623\tabularnewline
\hline 
$\xi$ & A001652 & A001571 & A077262 & A077291 & A077401 & A336625\tabularnewline
\hline 
$T_{t}$ & A075528 & A076139 & A077260 & A077289 & A077399 & A336624\tabularnewline
\hline 
$T_{\xi}$ & A029549 & A076140 & A077261 & A077290 & A077400 & A336626\tabularnewline
\hline 
\end{tabular}
\end{table}

Among all solutions, $t=0$ is always a first solution of (\ref{eq:1})
for all non-square integer value of $k$, yielding $\xi=0$. 

Let's consider the two cases of $k=2$ and $k=7$ yielding the successive
solution pairs as shown in Table \ref{tab2:Solutions-of-}. We indicate
also the ratios $t_{n}/t_{n-1}$ for both cases and $t_{n}/t_{n-2}$
for $k=7$. It is seen that for $k=2$, the ratio $t_{n}/t_{n-1}$
varies between close values, from 7 down to 5.829, while for $k=7$,
the ratio $t_{n}/t_{n-1}$ alternates between values 2.5 ... 2.216
and 7.8 ... 7.23, while the ratio $t_{n}/t_{n-2}$ decreases regularly
from 19.5 to 16.023 (corresponding approximately to the product of
the alternating values of the ratio $t_{n}/t_{n-1}$). We call rank
$r$ the integer value such that $t_{n}/t_{n-r}$ is approximately
constant or, better, decreases regularly without jumps (a more precise
definition is given further). So, here, the case $k=2$ has rank $r=1$
and the case $k=7$ has rank $r=2$.

\begin{table}

\caption{\label{tab2:Solutions-of-}Solutions of (\ref{eq:1}) for $k=2,7$}

\centering{}%
\begin{tabular}{|c|rrl|rrll|}
\hline 
{\small{}$n$} & \multicolumn{3}{c|}{{\small{}$k=2$}} & \multicolumn{4}{c|}{{\small{}$k=7$}}\tabularnewline
\cline{2-8} \cline{3-8} \cline{4-8} \cline{5-8} \cline{6-8} \cline{7-8} \cline{8-8} 
 & \multicolumn{1}{r|}{{\small{}$t_{n}$}} & \multicolumn{1}{r|}{{\small{}$\xi_{n}$}} & {\small{}$\frac{t_{n}}{t_{n-1}}$} & \multicolumn{1}{r|}{{\small{}$t_{n}$}} & \multicolumn{1}{r|}{{\small{}$\xi_{n}$}} & \multicolumn{1}{l|}{{\small{}$\frac{t_{n}}{t_{n-1}}$}} & {\small{}$\frac{t_{n}}{t_{n-2}}$}\tabularnewline
\hline 
\hline 
{\small{}0} & {\small{}0} & {\small{}0} &  & {\small{}0} & {\small{}0} &  & \tabularnewline
\hline 
{\small{}1} & {\small{}2} & {\small{}3} & {\small{}--} & {\small{}2} & 6 & {\small{}--} & {\small{}--}\tabularnewline
\hline 
{\small{}2} & {\small{}14} & {\small{}20} & {\small{}7} & 5 & {\small{}14} & 2.5 & {\small{}--}\tabularnewline
\hline 
{\small{}3} & {\small{}84} & {\small{}119} & {\small{}6} & {\small{}39} & {\small{}104} & 7.8 & 19.5\tabularnewline
\hline 
{\small{}4} & {\small{}492} & {\small{}696} & {\small{}5.857} & {\small{}87} & {\small{}231} & 2.231 & 17.4\tabularnewline
\hline 
{\small{}5} & {\small{}2870} & {\small{}4059} & {\small{}5.833} & {\small{}629} & {\small{}1665} & 7.230 & 16.128\tabularnewline
\hline 
{\small{}6} & {\small{}16730} & {\small{}23660} & {\small{}5.829} & {\small{}1394} & {\small{}3689} & 2.216 & 16.023\tabularnewline
\hline 
\end{tabular}
\end{table}
In \cite{key-9},we showed that the rank $r$ is the index of $t_{r}$
and $\xi_{r}$ solutions of (\ref{eq:1}) such that 

\begin{equation}
\kappa=t_{r}+t_{r-1}=\xi_{r}-\xi_{r-1}-1\label{eq:3.2}
\end{equation}
and that the ratio $t_{2r}/t_{r}$, corrected by the ratio $t_{r-1}/t_{r}$,
is equal to a constant $2\kappa+3$
\begin{equation}
\frac{t_{2r}-t_{r-1}}{t_{r}}=2\kappa+3\label{eq:3-0}
\end{equation}
For example, for $k=7$ and $r=2$, (\ref{eq:3.2}) and (\ref{eq:3-0})
yield respectively, $\kappa=7$ and $2\kappa+3=17$. 

Four recurrent equations for $t_{n},\xi_{n},T_{t_{n}}$ and $T_{\xi_{n}}$
are given in \cite{key-9} for each non-square integer value of $k$ 

\begin{align}
t_{n} & =2\left(\kappa+1\right)t_{n-r}-t_{n-2r}+\kappa\label{eq:3.3}\\
\xi_{n} & =2\left(\kappa+1\right)\xi_{n-r}-\xi_{n-2r}+\kappa\label{eq:3.3-1}\\
T_{t_{n}} & =\left(4\left(\kappa+1\right)^{2}-2\right)T_{t_{n-r}}-T_{t_{n-2r}}+\left(T_{\kappa}-\gamma\right)\label{eq:3.3-2}\\
T_{\xi_{n}} & =\left(4\left(\kappa+1\right)^{2}-2\right)T_{\xi_{n-r}}-T_{\xi_{n-2r}}+k\left(T_{\kappa}-\gamma\right)\label{eq:3.3-3}
\end{align}
where coefficients are functions of two constants $\kappa$ and $\gamma$,
respectively the sum $\kappa$ and the product $\gamma=t_{r-1}t_{r}$
of the first two sequential values of $t_{r}$ and $t_{r-1}$. Note
that the first three relations (\ref{eq:3.3}) to (\ref{eq:3.3-2})
are independent from the value of $k$. 

\section{Congruence of $\xi$ modulo $k$\label{sec3:Congruence-of-}}

We use the following notations: for $A,B,C\in\mathbb{Z},B<C,C>1$,
$A\equiv B\left(\text{mod\,}C\right)$ means that $\exists D\in\mathbb{Z}$
such that $A=DC+B$, where $B$ and $C$ are called respectively the
remainder and the modulus. To search numerically for the values of
$t_{n}$ and $\xi_{n}$ that solve (\ref{eq:1}), one can use the
congruent properties of $\xi\left(\text{mod\,}\ensuremath{k}\right)$
given in the following propositions. In other words, we search in
the following propositions for expressions of the remainders in function
of $k$ or of its factors.
\begin{prop}
\label{prop1:For-,-}For $\forall s,k\in\mathbb{Z}^{+}$, $k$ non-square,
$\exists\xi,\mu,\upsilon,i,j\in\mathbb{Z}^{+}$, such that if $\xi_{i}$
are solutions of (\ref{eq:1}), then for $\xi_{i}\equiv\mu_{j}\left(\text{mod\,}k\right)$
with $1\leq j\leq\upsilon$, the number $\upsilon$ of remainders
is always even, $\upsilon\equiv0\left(\text{mod\,}2\right)$, the
remainders come in pairs whose sum is always equal to $\left(k-1\right)$,
and the sum of all remainders is always equal to the product of $\left(k-1\right)$
and the number of remainder pairs, $\sum_{j=1}^{\upsilon}\mu_{j}=\left(k-1\right)\upsilon/2$.
\end{prop}

\begin{proof}
Let $s,i,j,k,\xi,\mu,\upsilon,\alpha,\beta\in\mathbb{Z}^{+}$, $k$
non-square, and $\xi_{i}$ solutions of (\ref{eq:1}). Rewriting (\ref{eq:1})
as $T_{t_{i}}=T_{\xi_{i}}/k$, for $T_{t_{i}}$ to be integer, $k$
must divide exactly $T_{\xi_{i}}=\xi_{i}\left(\xi_{i}+1\right)/2$,
i.e., among all possibilities, $k$ divides either $\xi_{i}$ or $\left(\xi_{i}+1\right)$,
yielding two possible solutions $\xi_{i}\equiv0\left(\text{mod\,}k\right)$
or $\xi_{i}\equiv-1\left(\text{mod}\,k\right)$, i.e. $\upsilon=2$
and the set of $\mu_{j}$ includes $\left\{ 0,\left(k-1\right)\right\} $.
This means that $\xi_{i}$ are always congruent to either $0$ or
$\left(k-1\right)$ modulo $k$ for all non-square values of $k$.

Furthermore, if some $\xi_{i}$ are congruent to $\alpha$ modulo
$k$, then other $\xi_{i}$ are also congruent to $\beta$ modulo
$k$ with $\beta=\left(k-\alpha-1\right)$. As $\xi_{i}\equiv\alpha\left(\text{mod}\,k\right)$,
then $\xi_{i}\left(\xi_{i}+1\right)/2\equiv\left(\alpha\left(\alpha+1\right)/2\right)\left(\text{mod\,}k\right)$
and replacing $\alpha$ by $\alpha=\left(k-\beta-1\right)$ yields
$\left(\alpha\left(\alpha+1\right)/2\right)=\left(\left(k-\beta-1\right)\left(k-\beta\right)/2\right)$,
giving $\xi_{i}\left(\xi_{i}+1\right)/2\equiv\left(\left(k-\beta-1\right)\left(k-\beta\right)/2\right)\left(\text{mod\,}k\right)\equiv$ 

$\left(\beta\left(\beta+1\right)/2\right)\left(\text{mod\,}k\right)$.
In this case, $\upsilon=4$ and the set of $\mu_{j}$ includes, but
not necessarily limits to, $\left\{ 0,\alpha,\left(k-\alpha-1\right),\left(k-1\right)\right\} $.
\end{proof}
Note that in some cases, $\upsilon>4$, as for $k=66,70,78,105,...$
, $\nu=8$. However, in some other cases, $\upsilon=2$ only and the
set of $\mu_{j}$ contains only $\left\{ 0,\left(k-1\right)\right\} $,
as shown in the next proposition. In this proposition, several rules
(R) are given constraining the congruence characteristics of $\xi_{i}$.
\begin{prop}
\label{prop2:For-,-}For $\forall s,k,\alpha,n\in\mathbb{Z}^{+}$,
$k$ non-square, $\alpha>1$, $\exists\xi,\mu,\upsilon,i\in\mathbb{Z}^{+}$,
such that if $\xi_{i}$ are solutions of (\ref{eq:1}), then $\xi_{i}$
are always only congruent to $0$ and $\left(k-1\right)$ modulo $k$
, and $\upsilon=2$ if either (R1) $k$ is prime, or (R2) $k=\alpha^{n}$
with $\alpha$ prime and $n$ odd, or (R3) $k=s^{2}+1$ with $s$
even, or (R4) $k=s^{\prime2}-1$ or (R5) $k=s^{\prime2}-2$ with $s^{\prime}$
odd.
\end{prop}

\begin{proof}
Let $s,s^{\prime},k,\alpha>1,n,i,\xi,\mu,\upsilon\in\mathbb{Z}^{+}$,
$k$ non-square, and $\xi_{i}$ are solutions of (\ref{eq:1}). 

(R1)+(R2): If $k$ is prime or if $k=\alpha^{n}$ (with $\alpha$
prime and $n$ odd as $k$ is non-square), then, in both cases, $k$
can only divide either $\xi_{i}$ or $\left(\xi_{i}+1\right)$, yielding
the two congruences $\xi_{i}\equiv0\left(\text{mod\,}k\right)$ and
$\xi_{i}\equiv-1\left(\text{mod\,}k\right)$.

(R3): If $k=s^{2}+1$ with $s$ even, the rank $r$ is always $r=2$
\cite{key-11}, and the only two sets of solutions are 
\begin{align}
\left(t_{1},\xi_{1}\right) & =\left(s\left(s-1\right),\left(s^{2}+1\right)\left(s-1\right)\right)\label{eq:2.8}\\
\left(t_{2},\xi_{2}\right) & =\left(s\left(s+1\right),\left(s^{2}+1\right)\left(s+1\right)-1\right)\label{eq:2.9}
\end{align}
as can be easily shown. For $t_{1}$, forming
\begin{align*}
kT_{t_{1}} & =\frac{1}{2}\left(s^{2}+1\right)\left(s\left(s-1\right)\right)\left(s\left(s-1\right)+1\right)\\
 & =\frac{1}{2}\left[\left(s^{2}+1\right)\left(s-1\right)\right]\left[\left(s^{2}+1\right)\left(s-1\right)+1\right]=T_{\xi_{1}}
\end{align*}
which is the triangular number of $\xi_{1}$. One obtains similarly
$\xi_{2}$ from $t_{2}$. These two relations (\ref{eq:2.8}) and
(\ref{eq:2.9}) show respectively that $\xi_{1}$ is congruent to
$0$ modulo $k$ and $\xi_{2}$ is congruent to $\left(k-1\right)$
modulo $k$.

(R4) For $k=s^{\prime2}-1$ with $s^{\prime}$ odd, the rank $r=2$
\cite{key-11}, and the only two sets of solutions are
\begin{align}
\left(t_{1},\xi_{1}\right) & =\left(\left(s^{\prime}-1\right)s^{\prime}-1,\left(s^{\prime2}-1\right)\left(s^{\prime}-1\right)-1\right)\label{eq:2.13}\\
\left(t_{2},\xi_{2}\right) & =\left(\left(s^{\prime}-1\right)\left(s^{\prime}+2\right)+1,\left(s^{\prime2}-1\right)\left(s^{\prime}+1\right)\right)\label{eq:2.14}
\end{align}
as can be easily demonstrated as above. These two relations (\ref{eq:2.13})
and (\ref{eq:2.14}) show that $\xi_{1}$ and $\xi_{2}$ are congruent
respectively to $\left(k-1\right)$ and $0$ modulo $k$.

(R5) For $k=s^{\prime2}-2$ with $s^{\prime}$ odd, the rank $r=2$
\cite{key-11}, and the only two sets of solutions are
\begin{align}
\left(t_{1},\xi_{1}\right) & =\left(\frac{1}{2}\left(s^{\prime}-2\right)\left(s^{\prime}+1\right),\frac{1}{2}\left(s^{\prime2}-2\right)\left(s^{\prime}-1\right)-1\right)\label{eq:2.15}\\
\left(t_{2},\xi_{2}\right) & =\left(\frac{s^{\prime}}{2}\left(s^{\prime}+1\right)-1,\frac{1}{2}\left(s^{\prime2}-2\right)\left(s^{\prime}+1\right)\right)\label{eq.2.16}
\end{align}
as can easily be shown as above. These two relations (\ref{eq:2.15})
and (\ref{eq.2.16}) show that $\xi_{1}$ and $\xi_{2}$ are congruent
respectively to $\left(k-1\right)$ and $0$ modulo $k$.
\end{proof}
There are other cases of interest as shown in the next two Propositions
\begin{prop}
\label{prop3:For-,-,}For $\forall n\in\mathbb{Z}^{+}$, $\exists k,\xi,\mu<k,i,j\in\mathbb{Z}^{+}$,
$k$ non-square, such that if $\xi_{i}$ are solutions of (\ref{eq:1})
with $\xi_{i}\equiv\mu_{j}\left(\text{mod}\,k\right)$, and (R6) if
$k$ is twice a triangular number $k=n\left(n+1\right)=2T_{n}$, then
the set of $\mu_{j}$ includes $\left\{ 0,n,\left(n^{2}-1\right),\left(k-1\right)\right\} $,
with $1\leq j\leq\upsilon$.
\end{prop}

\begin{proof}
Let $n,k,\xi,\mu<k,i,j\in\mathbb{Z}^{+}$, $k$ non-square, and $\xi_{i}$
solutions of (\ref{eq:1}). Let $\xi_{i}\equiv\mu_{j}\left(\text{mod\,}k\right)$
with $1\leq j\leq\upsilon$. As the ratio $\xi_{i}\left(\xi_{i}+1\right)/k$
must be integer, $\xi_{i}\left(\xi_{i}+1\right)\equiv0\left(\text{mod\,}k\right)$
or $\mu_{j}\left(\mu_{j}+1\right)\equiv0\left(\text{mod}\,n\left(n+1\right)\right)$
which is obviously satisfied if $\mu_{j}=n$ or $\mu_{j}=\left(n^{2}-1\right)$.
\end{proof}
Finally, this last proposition gives a general expression of the congruence
$\xi_{i}\left(\text{mod\,}\ensuremath{k}\right)$ for most cases to
find the remainders $\mu_{j}$ other than $0$ and $\left(k-1\right)$.
\begin{prop}
\label{prop4:For-,-,}For $\forall n>1\in\mathbb{Z}^{+}$, $\exists k,f,\xi,\nu<n<k,\mu<k,m<n,i,j\in\mathbb{Z}^{+}$,
$k$ non-square, let $\xi_{i}$ be solutions of (\ref{eq:1}) with
$\xi_{i}\equiv\mu_{j}\left(\text{mod}\,k\right)$, let $f$ be a factor
of $k$ such that $f=k/n$ with $f\equiv\nu\left(\text{mod\,}\ensuremath{n}\right)$
and $k\equiv\nu n\left(\text{mod}\,n^{2}\right)$, then the set of
$\mu_{j}$ includes either $\left\{ 0,mf,\left(\left(n-m\right)f-1\right),\left(k-1\right)\right\} $
or $\left\{ 0,\left(mf-1\right),\left(n-m\right)f,\left(k-1\right)\right\} $,
where $m$ is an integer multiplier of $f$ in the congruence relation
and such that $m<n/2$ or $m<\left(n+1\right)/2$ for $n$ being even
or odd respectively, and $1\leq j\leq\upsilon$.
\end{prop}

\begin{proof}
Let $n>1,k,f,\xi,\mu<k,m<n,i,j<n<k\in\mathbb{Z}^{+}$, $k$ non-square,
and $\xi_{i}$ a solution of (\ref{eq:1}). Let $\xi_{i}\equiv\mu_{j}\left(\text{mod\,}k\right)$
with $1\leq j\leq\upsilon$. As the ratio $\xi_{i}\left(\xi_{i}+1\right)/k$
must be integer, $\xi_{i}\left(\xi_{i}+1\right)\equiv0\left(\text{mod\,}k\right)$
or $\mu_{j}\left(\mu_{k}+1\right)\equiv0\left(\text{mod\,}fn\right)$.
For a proper choice of the factor $f$ of $k$, let $\mu_{j}$ be
a multiple of $f$, $\mu_{j}=mf$, then $m\left(mf+1\right)\equiv0\left(\text{mod\,}n\right)$.
As $f\equiv\nu\left(\text{mod\,}\ensuremath{n}\right)$, one has 
\begin{equation}
m\left(m\nu+1\right)\equiv0\left(\text{mod}\,n\right)\label{eq:120}
\end{equation}
. Let now $\left(\mu_{j}+1\right)$ be a multiple of $f$, $\mu_{j}+1=mf$,
then $m\left(mf-1\right)\equiv0\left(\text{mod\,}n\right)$ or
\begin{equation}
m\left(m\nu-1\right)\equiv0\left(\text{mod\,}n\right)\label{eq:121}
\end{equation}

An appropriate combination of integer parameters $m$ and $\nu$ guarantees
that (\ref{eq:120}) and (\ref{eq:121}) are satisfied. Proposition
1 yields the other remainder value as $mf+\left(n-m\right)f-1=k-1$
and $\left(mf-1\right)+\left(n-m\right)f=k-1$.
\end{proof}
The appropriate combinations of integer parameters $m$ and $\nu$
are given in Table \ref{tab3:Combination-of-parameters} for $2\leq n\leq12$.
The sign $-$ in subscript corresponds to the remainder $\left(mf-1\right)$;
the sign $/$ indicates an absence of combination.

\begin{table}
\caption{\label{tab3:Combination-of-parameters}Combination of parameters $m$
and $\nu$ for $2\protect\leq n\protect\leq12$}

\centering{}%
\begin{tabular}{|c|c|ccccccccccc|}
\cline{3-13} \cline{4-13} \cline{5-13} \cline{6-13} \cline{7-13} \cline{8-13} \cline{9-13} \cline{10-13} \cline{11-13} \cline{12-13} \cline{13-13} 
\multicolumn{1}{c}{$m$} &  & \multicolumn{11}{c|}{$\nu$}\tabularnewline
\cline{3-13} \cline{4-13} \cline{5-13} \cline{6-13} \cline{7-13} \cline{8-13} \cline{9-13} \cline{10-13} \cline{11-13} \cline{12-13} \cline{13-13} 
\multicolumn{1}{c}{} & $\searrow$ & \multicolumn{1}{c|}{1} & \multicolumn{1}{c|}{2} & \multicolumn{1}{c|}{3} & \multicolumn{1}{c|}{4} & \multicolumn{1}{c|}{5} & \multicolumn{1}{c|}{6} & \multicolumn{1}{c|}{7} & \multicolumn{1}{c|}{8} & \multicolumn{1}{c|}{9} & \multicolumn{1}{c|}{10} & 11\tabularnewline
\hline 
\multirow{11}{*}{$n$} & 2 & 1\_ &  &  &  &  &  &  &  &  &  & \tabularnewline
\cline{2-4} \cline{3-4} \cline{4-4} 
 & 3 & 1\_ & 1 &  &  &  &  &  &  &  &  & \tabularnewline
\cline{2-5} \cline{3-5} \cline{4-5} \cline{5-5} 
 & 4 & 1\_ & / & 1 &  &  &  &  &  &  &  & \tabularnewline
\cline{2-6} \cline{3-6} \cline{4-6} \cline{5-6} \cline{6-6} 
 & 5 & 1\_ & 2 & 2\_ & 1 &  &  &  &  &  &  & \tabularnewline
\cline{2-7} \cline{3-7} \cline{4-7} \cline{5-7} \cline{6-7} \cline{7-7} 
 & 6 & 1\_ & / & / & / & 1 &  &  &  &  &  & \tabularnewline
\cline{2-8} \cline{3-8} \cline{4-8} \cline{5-8} \cline{6-8} \cline{7-8} \cline{8-8} 
 & 7 & 1\_ & 3 & 2 & 2\_ & 3\_ & 1 &  &  &  &  & \tabularnewline
\cline{2-9} \cline{3-9} \cline{4-9} \cline{5-9} \cline{6-9} \cline{7-9} \cline{8-9} \cline{9-9} 
 & 8 & 1\_ & / & 3\_ & / & 3 & / & 1 &  &  &  & \tabularnewline
\cline{2-10} \cline{3-10} \cline{4-10} \cline{5-10} \cline{6-10} \cline{7-10} \cline{8-10} \cline{9-10} \cline{10-10} 
 & 9 & 1\_ & 4 & / & 2 & 2\_ & / & 4\_ & 1 &  &  & \tabularnewline
\cline{2-11} \cline{3-11} \cline{4-11} \cline{5-11} \cline{6-11} \cline{7-11} \cline{8-11} \cline{9-11} \cline{10-11} \cline{11-11} 
 & 10 & 1\_ & / & 3 & / & 5\_ & / & 3\_ & / & 1 &  & \tabularnewline
\cline{2-12} \cline{3-12} \cline{4-12} \cline{5-12} \cline{6-12} \cline{7-12} \cline{8-12} \cline{9-12} \cline{10-12} \cline{11-12} \cline{12-12} 
 & 11 & 1\_ & 5 & 4\_ & 3\_ & 2 & 2\_ & 3 & 4 & 5\_ & 1 & \tabularnewline
\cline{2-13} \cline{3-13} \cline{4-13} \cline{5-13} \cline{6-13} \cline{7-13} \cline{8-13} \cline{9-13} \cline{10-13} \cline{11-13} \cline{12-13} \cline{13-13} 
 & 12 & 1\_ & / & / & / & 3 & / & 4\_ & / & / & / & 1\tabularnewline
\hline 
\end{tabular}
\end{table}

One deduces from Table \ref{tab3:Combination-of-parameters} the following
simple rules:

1) $\forall n\in\mathbb{Z}^{+}$, only those values of $\nu$ that
are co-prime with $n$ must be kept, all other combinations (indicated
by $/$ in Table \ref{tab3:Combination-of-parameters}) must be discarded
as they correspond to combinations with smaller values of $n$ and
$\nu$; for $n$ even, this means that all even values of $\nu$ must
be discarded. For example, $\nu=2$ and $n=4$ are not co-prime and
their combination obviously corresponds to $\nu=1$ and $n=2$.

2) For $\nu=1$ and $\nu=n-1$, all values of $m$ are $m=1$ with
respectively the remainders $\left(mf-1\right)$ and $mf$.

3) For $\forall n,i\in\mathbb{Z}^{+}$, $n$ odd, $2\leq i\leq\left(n-1\right)/2$,
and for $\nu=\left(n-\left(2i-3\right)\right)/2$ and $\nu=\left(n+\left(2i-3\right)\right)/2$,
all the values of $m$ are $m=i$.

4) For $\forall n\in\mathbb{Z}^{+}$, $n$ odd, and for $\nu=2$ and
$\nu=n-2$, the remainders are respectively $mf$ and $\left(mf-1\right)$.

5) For $\forall n,i\in\mathbb{Z}^{+}$, $n$ even, $2\leq i\leq n/2$,
and for $\nu=\left(n-\left(2i-3\right)\right)/2$ and $\nu=\left(n+\left(2i-3\right)\right)/2$,
all the values of $m$ are $m=i$.

Expressions of $\mu_{i}$ are given in Table \ref{tab4:Expressions-of-}
for $2\leq n\leq12$ (with codes E$n\nu$). For example, for $k\equiv12\nu\left(\text{mod}\,12^{2}\right)$
and $\nu=5$ (code E125), i.e. $k=60,204,348,...$, $\xi_{i}\equiv\mu_{j}\left(\text{mod\,}k\right)$
with the set of remainders $\mu_{j}$ including $\left\{ 0,mf,\left(\left(n-m\right)f-1\right),\left(k-1\right)\right\} $
with $m=3$ (see Table \ref{tab3:Combination-of-parameters}) and
$f=k/12=5,17,29...$respectively.

\begin{table}
\caption{\label{tab4:Expressions-of-}Expressions of $\mu_{j}$ for $2\protect\leq n\protect\leq12$}

\centering{}%
\begin{tabular}{|c|c|c|c|c|c|c|}
\hline 
$n$ & $\nu$ & $m$ & $k\equiv$ & $f$ & $\mu_{j}$ & Code\tabularnewline
\hline 
\hline 
2 & 1 & 1 & $2\left(\text{mod\,}4\right)$ & $k/2$ & $0,(k/2)-1,k/2,k-1$ & E21\tabularnewline
\hline 
3 & 1 & 1 & $3\left(\text{mod\,}9\right)$ & $k/3$ & $0,\left(k/3\right)-1,2k/3,k-1$ & E31\tabularnewline
 & 2 & 1 & $6\left(\text{mod}\,9\right)$ &  & $0,k/3,\left(2k/3\right)-1,k-1$ & E32\tabularnewline
\hline 
4 & 1 & 1 & $4\left(\text{mod\,}16\right)$ & $k/4$ & $0,\left(k/4\right)-1,3k/4,k-1$ & E41\tabularnewline
 & 3 & 1 & $12\left(\text{mod\,}16\right)$ &  & $0,k/4,\left(3k/4\right)-1,k-1$ & E43\tabularnewline
\hline 
5 & 1 & 1 & $5\left(\text{mod\,}25\right)$ & $k/5$ & $0,\left(k/5\right)-1,4k/5,k-1$ & E51\tabularnewline
 & 2 & 2 & $10\left(\text{mod\,}25\right)$ &  & $0,2k/5,\left(3k/5\right)-1,k-1$ & E52\tabularnewline
 & 3 & 2 & $15\left(\text{mod\,}25\right)$ &  & $0,\left(2k/5\right)-1,3k/5,k-1$ & E53\tabularnewline
 & 4 & 1 & $20\left(\text{mod\,}25\right)$ &  & $0,k/5,\left(4k/5\right)-1,k-1$ & E54\tabularnewline
\hline 
6 & 1 & 1 & $6\left(\text{mod\,}36\right)$ & $k/6$ & $0,\left(k/6\right)-1,5k/6,k-1$ & E61\tabularnewline
 & 5 & 1 & $30\left(\text{mod\,}36\right)$ &  & $0,k/6,\left(5k/6\right)-1,k-1$ & E65\tabularnewline
\hline 
7 & 1 & 1 & $7\left(\text{mod\,}49\right)$ & $k/7$ & $0,\left(k/7\right)-1,6k/7,k-1$ & E71\tabularnewline
 & 2 & 2 & $14\left(\text{mod\,}49\right)$ &  & $0,3k/7,\left(4k/7\right)-1,k-1$ & E72\tabularnewline
 & 3 & 3 & $21\left(\text{mod\,}49\right)$ &  & $0,2k/7,\left(5k/7\right)-1,k-1$ & E73\tabularnewline
 & 4 & 3 & $28\left(\text{mod\,}49\right)$ &  & $0,\left(2k/7\right)-1,5k/7,k-1$ & E74\tabularnewline
 & 5 & 2 & $35\left(\text{mod\,}49\right)$ &  & $0,\left(3k/7\right)-1,4k/7,k-1$ & E75\tabularnewline
 & 6 & 1 & $42\left(\text{mod\,}49\right)$ &  & $0,k/7,\left(6k/7\right)-1,k-1$ & E76\tabularnewline
\hline 
8 & 1 & 1 & $8\left(\text{mod\,}64\right)$ & $k/8$ & $0,\left(k/8\right)-1,7k/8,k-1$ & E81\tabularnewline
 & 3 & 3 & $24\left(\text{mod\,}64\right)$ &  & $0,\left(3k/8\right)-1,5k/8,k-1$ & E83\tabularnewline
 & 5 & 3 & $40\left(\text{mod\,}64\right)$ &  & $0,3k/8,\left(5k/8\right)-1,k-1$ & E85\tabularnewline
 & 7 & 1 & $56\left(\text{mod\,}64\right)$ &  & $0,k/8,\left(7k/8\right)-1,k-1$ & E87\tabularnewline
\hline 
9 & 1 & 1 & $9\left(\text{mod\,}81\right)$ & $k/9$ & $0,(k/9)-1,8k/9,k-1$ & E91\tabularnewline
 & 2 & 4 & $18\left(\text{mod\,}81\right)$ &  & $0,4k/9,(5k/9)-1,k-1$ & E92\tabularnewline
 & 4 & 2 & $36\left(\text{mod\,}81\right)$ &  & $0,2k/9,(7k/9)-1,k-1$ & E94\tabularnewline
 & 5 & 2 & $45\left(\text{mod\,}81\right)$ &  & $0,(2k/9)-1,7k/9,k-1$ & E95\tabularnewline
 & 7 & 4 & $63\left(\text{mod\,}81\right)$ &  & $0,(4k/9)-1,5k/9,k-1$ & E97\tabularnewline
 & 8 & 1 & $72\left(\text{mod\,}81\right)$ &  & $0,k/9,(8k/9)-1,k-1$ & E98\tabularnewline
\hline 
10 & 1 & 1 & $10\left(\text{mod\,}100\right)$ & $k/10$ & $0,(k/10)-1,9k/10,k-1$ & E101\tabularnewline
 & 3 & 3 & $30\left(\text{mod\,}100\right)$ &  & $0,3k/10,(7k/10)-1,k-1$ & E103\tabularnewline
 & 7 & 3 & $70\left(\text{mod\,}100\right)$ &  & $0,(3k/10)-1,7k/10,k-1$ & E107\tabularnewline
 & 9 & 1 & $90\left(\text{mod}\,100\right)$ &  & $0,k/10,(9k/10)-1,k-1$ & E109\tabularnewline
\hline 
11 & 1 & 1 & $11\left(\text{mod\,}121\right)$ & $k/11$ & $0,(k/11)-1,10k/11,k-1$ & E111\tabularnewline
 & 2 & 5 & $22\left(\text{mod\,}121\right)$ &  & $0,5k/11,(6k/11)-1,k-1$ & E112\tabularnewline
 & 3 & 4 & $33\left(\text{mod\,}121\right)$ &  & $0,(4k/11)-1,7k/11,k-1$ & E113\tabularnewline
 & 4 & 3 & $44\left(\text{mod\,}121\right)$ &  & $0,(3k/11)-1,8k/11,k-1$ & E114\tabularnewline
 & 5 & 2 & $55\left(\text{mod\,}121\right)$ &  & $0,2k/11,(9k/11)-1,k-1$ & E115\tabularnewline
 & 6 & 2 & $66\left(\text{mod\,}121\right)$ &  & $0,(2k/11)-1,9k/11,k-1$ & E116\tabularnewline
 & 7 & 3 & $77\left(\text{mod\,}121\right)$ &  & $0,3k/11,(8k/11)-1,k-1$ & E117\tabularnewline
 & 8 & 4 & $88\left(\text{mod\,}121\right)$ &  & $0,4k/11,(7k/11)-1,k-1$ & E118\tabularnewline
 & 9 & 5 & $99\left(\text{mod\,}121\right)$ &  & $0,(5k/11)-1,6k/11,k-1$ & E119\tabularnewline
 & 10 & 1 & $110\left(\text{mod\,}121\right)$ &  & $0,k/11,(10k/11)-1,k-1$ & E1110\tabularnewline
\hline 
12 & 1 & 1 & $12\left(\text{mod\,}144\right)$ & $k/12$ & $0,(k/12)-1,11k/12,k-1$ & E121\tabularnewline
 & 5 & 3 & $60\left(\text{mod\,}144\right)$ &  & $0,3k/12,(9k/12)-1,k-1$ & E125\tabularnewline
 & 7 & 4 & $84\left(\text{mod\,}144\right)$ &  & $0,(4k/12)-1,8k/12,k-1$ & E127\tabularnewline
 & 11 & 1 & $132\left(\text{mod\,}144\right)$ &  & $0,k/12,(11k/12)-1,k-1$ & E1211\tabularnewline
\hline 
\end{tabular}
\end{table}

Values of the remainders $\mu_{j}$ are given in Table \ref{tab5:Values-of-}
for $2\leq k\leq120$, with rule (R) and expression (E) codes as references.
R and E codes separated by comas imply that all references apply simultaneously
to the case; E codes separated by + mean that all expressions are
applicable to the case; some expression references are sometimes missing.
One observes that in two cases (for $k=74$ and 104), expressions
could not be found (indicated by question marks).

\begin{table}
\caption{\label{tab5:Values-of-}Values of $\mu_{j}$ for $2\protect\leq k\protect\leq120$}

\centering{}%
\begin{tabular}{|c|c|c||c|c|c|}
\hline 
$k$ & $\mu_{j}$ & References & $k$ & $\mu_{j}$ & References\tabularnewline
\hline 
\hline 
2 & 0,1 & R1,R6,E21 & 63 & 0,27,35,62 & E72,E97\tabularnewline
\hline 
3 & 0,2 & R1,E31 & 65 & 0,64 & R3\tabularnewline
\hline 
5 & 0,4 & R1,R3,E51 & 66 & 0,11,21,32,33,44,54,65 & E21+E31+E65+E116\tabularnewline
\hline 
6 & 0,2,3,5 & R6,E21,E32,E61 & 67 & 0,66 & R1\tabularnewline
\hline 
7 & 0,6 & R1,R5,E71 & 68 & 0,16,51,67 & E41\tabularnewline
\hline 
8 & 0,7 & R2,R4,E81 & 69 & 0,23,45,68 & E32\tabularnewline
\hline 
10 & 0,4,5,9 & E21,E52,E101 & 70 & 0,14,20,34,35,49,55,69 & E21+E54+E73+E107\tabularnewline
\hline 
11 & 0,10 & R1,E111 & 71 & 0,70 & R1\tabularnewline
\hline 
12 & 0,3,8,11 & R6,E31,E43,E121 & 72 & 0,8,63,71 & R6,E81,E98\tabularnewline
\hline 
13 & 0,12 & R1 & 73 & 0,72 & R1\tabularnewline
\hline 
14 & 0,6,7,13 & E21,E72 & 74 & 0,73 & ? \tabularnewline
\hline 
15 & 0,5,9,14 & E32,E53 & 75 & 0,24,50,74 & E31\tabularnewline
\hline 
17 & 0,16 & R1,R3 & 76 & 0,19,56,75 & E43\tabularnewline
\hline 
18 & 0,8,9,17 & E21,E92 & 77 & 0,21,55,76 & E74,E117\tabularnewline
\hline 
19 & 0,18 & R1 & 78 & 0,12,26,38,39,51,65,77 & E21+E32+E61\tabularnewline
\hline 
20 & 0,4,15,19 & R6,E41,E54 & 79 & 0,78 & R1,R5\tabularnewline
\hline 
21 & 0,6,14,20 & E31,E73 & 80 & 0,79 & R4\tabularnewline
\hline 
22 & 0,10,11,21 & E21,E112 & 82 & 0,40,41,81 & E21\tabularnewline
\hline 
23 & 0,22 & R1,R5 & 83 & 0,82 & R1\tabularnewline
\hline 
24 & 0,23 & R4 & 84 & 0,27,56,83 & E31,E127\tabularnewline
\hline 
26 & 0,12,13,25 & E21 & 85 & 0,34,50,84 & E52\tabularnewline
\hline 
27 & 0,26 & R2 & 86 & 0,42,43,85 & E21\tabularnewline
\hline 
28 & 0,7,20,27 & E43,E74 & 87 & 0,29,57,86 & E32\tabularnewline
\hline 
29 & 0,28 & R1 & 88 & 0,32,55,87 & E83,E118\tabularnewline
\hline 
30 & 0,5,24,29 & R6,E51,E65 & 89 & 0,88 & R1\tabularnewline
\hline 
31 & 0,30 & R1 & 90 & 0,9,80,89 & R6,E91,E109\tabularnewline
\hline 
32 & 0,31 & R2 & 91 & 0,13,77,90 & E75\tabularnewline
\hline 
33 & 0,11,21,32 & E32,E113 & 92 & 0,23,68,91 & E43\tabularnewline
\hline 
34 & 0,16,17,33 & E21 & 93 & 0,30,62,92 & E31\tabularnewline
\hline 
35 & 0,14,20,34 & E52,E75 & 94 & 0,46,47,93 & E21\tabularnewline
\hline 
37 & 0,36 & R1,R3 & 95 & 0,19,75,94 & E54\tabularnewline
\hline 
38 & 0,18,19,37 & E21 & 96 & 0,32,63,95 & E32\tabularnewline
\hline 
39 & 0,12,26,38 & E31 & 97 & 0,96 & R1\tabularnewline
\hline 
40 & 0,15,24,39 & E53,E85 & 98 & 0,48,49,97 & E21\tabularnewline
\hline 
41 & 0,40 & R1 & 99 & 0,44,54,98 & E92,E119\tabularnewline
\hline 
42 & 0,6,35,41 & R6,E61,E76 & 101 & 0,100 & R1,R3\tabularnewline
\hline 
43 & 0,42 & R1 & 102 & 0,50,51,102 & E21\tabularnewline
\hline 
44 & 0,11,32,43 & E43,E114 & 103 & 0,102 & R1\tabularnewline
\hline 
45 & 0,9,35,44 & E54,E95 & 104 & 0,103 & ?\tabularnewline
\hline 
46 & 0,22,23,245 & E21 & 105 & 0,14,20,35,69,84,90,104 & E32+E51+E71\tabularnewline
\hline 
47 & 0,46 & R1,R5 & 106 & 0,52,53,105 & E21\tabularnewline
\hline 
48 & 0,47 & R4 & 107 & 0,106 & R1\tabularnewline
\hline 
50 & 0,24,25,49 & E21 & 108 & 0,27,80,107 & E43\tabularnewline
\hline 
51 & 0,17,33,50 & E32 & 109 & 0,108 & R1\tabularnewline
\hline 
52 & 0,12,39,51 & E41 & 110 & 0,10,99,109 & R6,E101,E1110\tabularnewline
\hline 
53 & 0,52 & R1 & 111 & 0,36,74,110 & E31\tabularnewline
\hline 
54 & 0,26,27,53 & E21 & 112 & 0,48,63,111 & E72\tabularnewline
\hline 
55 & 0,10,44,54 & E51,E115 & 113 & 0,112 & R1\tabularnewline
\hline 
56 & 0,7,48,55 & R6,E71,E87 & 114 & 0,56,57,113 & E21\tabularnewline
\hline 
57 & 0,18,38,56 & E31 & 115 & 0,45,69,114 & E53\tabularnewline
\hline 
58 & 0,28,29,57 & E21 & 116 & 0,28,87,115 & E41\tabularnewline
\hline 
59 & 0,58 & R1 & 117 & 0,26,90,116 & E94\tabularnewline
\hline 
60 & 0,15,44,59 & E43,E125 & 118 & 0,58,59,117 & E21\tabularnewline
\hline 
61 & 0,60 & R1 & 119 & 0,118 & R1,R5 \tabularnewline
\hline 
62 & 0,30,31,61 & E21 & 120 & 0,15,104,119 & E87\tabularnewline
\hline 
\end{tabular}
\end{table}

This Table \ref{tab5:Values-of-} gives correctly the values of the
remainder pairs in most of the cases. There are although some exceptions
and some values missing.

Among the exceptions to the values given in Table \ref{tab5:Values-of-},
for $n=2$, remainders values for $k=30,42,74,90,110,\ldots$ are
different from the theoretical ones in Table \ref{tab4:Expressions-of-}.
Furthermore, for $k=66,70,78,105,...$, additional remainders exist.
Expressions are missing for $k=74$ (E21) and 104 (E85). Finally,
one observes also that for 16 cases, some Rules or Expressions supersede
some other Expressions (indicated by Ra > Exy or Exy > Ezt), as reported
in Table \ref{tab6:Rules-and-Expressions}. For example, Rule 6 supersedes
Expression 21 (R6 > E21) for $k=30,42,90,110$, i.e., $k=2T_{5},2T_{6},2T_{9},2T_{10},...$
and more generally for all $k=2T_{i}$ for $i\equiv1,2\left(\text{mod}4\right)$.

\begin{table}
\caption{\label{tab6:Rules-and-Expressions}Rules and Expressions superseding
other Rules and Expressions}

\centering{}%
\begin{tabular}{cl}
\hline 
$k$ & \tabularnewline
\hline 
\hline 
24 & R4 > E32; R4 > E83\tabularnewline
\hline 
30 & R6 > E21; R6 > E31; R6 > E103; E51 > E103; E65 > E103\tabularnewline
\hline 
42 & R6 > E21; R6 > E32\tabularnewline
\hline 
48 & R4 > E31\tabularnewline
\hline 
56 & R6 > E43\tabularnewline
\hline 
60 & E43 > E32; E43 > E52\tabularnewline
\hline 
65 & R3 > E53\tabularnewline
\hline 
72 & R6 > E43\tabularnewline
\hline 
80 & R4 > E51\tabularnewline
\hline 
84 & E31 > E41; E31 > E75\tabularnewline
\hline 
90 & R6 > E21; R6 > E53\tabularnewline
\hline 
102 & E21 > E31; E21 > E65\tabularnewline
\hline 
110 & R6 > E21; R6 > E52\tabularnewline
\hline 
114 & E21 > E32; E21 > E61\tabularnewline
\hline 
119 & R1 > E73; R5 > E73\tabularnewline
\hline 
120 & E87 > R4; E87 > E31; E87 > E54\tabularnewline
\hline 
\end{tabular}
\end{table}

Note that 11 of these 16 values of $k$ are multiple of 6, the others
are 2 mod 6 and 5 mod 6 for, respectively three and two cases. One
notices as well, that generally, Ra and Exy supersede Ezt with $x<z$
and $t<y$, except for $k=60$ and $120$.

\section{Conclusions\label{sec4:Conclusions}}

We have shown that, for indices $\xi$ of triangular numbers multiples
of other triangular numbers, the remainders in the congruence relations
of $\xi$ modulo $k$ come always in pairs whose sum always equal
$\left(k-1\right)$, always include 0 and $\left(k-1\right)$, and
only 0 and $\left(k-1\right)$ if $k$ is prime, or an odd power of
a prime, or an even square plus one or an odd square minus one or
minus two. If the multiplier $k$ is twice a triangular number of
$n,$the set of remainders includes also $n$ and $\left(n^{2}-1\right)$
and if $k$ has integer factors, the set of remainders include multiple
of a factor following certain rules. Finally, algebraic expressions
are found for remainders in function of $k$ and its factors. Several
exceptions are noticed as well and it appears that there are superseding
rules between the various rules and expressions. 

This approach allows to eliminate in numerical searches those $\left(k-\upsilon\right)$
values of $\xi_{i}$ that are known not to provide solutions of (\ref{eq:1}),
where $\upsilon$ is the even number of remainders. The gain is typically
in the order of $k/\upsilon$, with $\upsilon\ll k$ for large values
of $k$.

\end{document}